\documentclass[runningheads,envcountsame,a4paper]{llncs} 

\usepackage{amssymb,amsmath,mathtools}
\usepackage{wasysym}
\usepackage{graphicx}
\usepackage{epsfig}
\usepackage{latexsym}
\usepackage[english]{babel}
\usepackage[utf8]{inputenc}
\usepackage{csquotes}
\usepackage{algorithm} 
\usepackage{algpseudocode}
\usepackage{tikz}
\usetikzlibrary{arrows,shapes.geometric,positioning,calc,cd}
\usepackage[colorinlistoftodos]{todonotes}
\usepackage{changepage}
\usepackage{subfig}
\usepackage{soul}
\usepackage{hyperref}
\usepackage{todonotes}
\usepackage{enumerate}

\newtheorem{thm}{Theorem}
\newtheorem{prop}[thm]{Proposition}
\newtheorem{cor}[thm]{Corollary}
\newtheorem{lem}[thm]{Lemma}

\newtheorem{exa}[thm]{Example}

\newtheorem{rmk}[thm]{Remark}
\newtheorem{obs}[thm]{Observation}

\newsavebox{\qedB}

\sbox{\qedB}{\setlength{\unitlength}{1mm}
 \begin{picture}(4,4)(0,0)
  \thinlines
  {\put(0,0){\framebox(2.83,2.83){}}}%
  {\put(1.17,1.17){\framebox(2.83,2.83){}}}%
  {\put(0,0){\framebox(4,4){}}}%
  {\put(1.17,1.17){{\rule{1ex}{1ex} }}}%
 \end{picture}}
 
\newcommand{\QEDB}{\ifmmode\def\next{\tag"\usebox{\qedB}"}%
 \else\let\next=\relax
 {\unskip\nobreak\hfil\penalty50
 \hskip2em\hbox{}\nobreak\hfil\usebox{\qedB}
 \parfillskip=0pt \finalhyphendemerits=0\penalty-100\bigskip}\fi\next}

\newcommand{\Alphabet}{\hbox{\rm Alph}}

\newcommand{\fac}{\hbox{\rm Fac}}

\newcommand{\Class}{\hbox{\rm Class}}

\newcommand{\prim}{\hbox{\rm Prim}}

\newcommand{\N}{\mathbb N}

\newcommand{\bprop}{\begin{prop}}
\newcommand{\eprop}{\end{prop}}
\newcommand{\bcor}{\begin{cor}}
\newcommand{\ecor}{\end{cor}}
\newcommand{\blem}{\begin{lem}}
\newcommand{\elem}{\end{lem}}

\newcommand{\cl}{\mathrm{Class}}
\newcommand{\pri}{\mathrm{Prim}}
\newcommand{\pow}{\mathrm{Power}}

\newcommand{\sq}{\mathrm{Sq}}
\newcommand{\NN}{\mathbb{N}}
\newcommand{\qcl}{\mathbb{Q}\mathrm{Class}}

\newcommand{\QQ}{\mathbb{Q}}
\newcommand{\ZZ}{\mathbb{Z}}

\newcommand{\rp}{\mathrm{RP}}
\newcommand{\indep}{\mathrm{Indep}}
\renewcommand{\sc}{\mathrm{Sc}}

\title{A Tighter Upper Bound for the Number of Distinct Squares in Circular Words}
\toctitle{Square complexity}

\authorrunning{S. Li, Y. Song}
\tocauthor{Shuo Li, Yuan Song}
\author{Shuo Li\inst{1} \and Yuan Song\inst{2}}

\institute{Hangzhou International Innovation Institute of Beihang University, \\Sino-French Laboratory for Mathematics\\
\email{shuoli@buaa.edu.cn}
\and
LMIB-School of Mathematical Sciences, Beihang University, Beijing\\
\email{yuan.song@buaa.edu.cn}
 }

\begin{document}

\maketitle
\begin{abstract}
A \emph{square} is a word of the form $uu$, where $u$ is a nonempty finite word.
Given a finite word $w$ of length $n$, let $[w]$ denote the corresponding \emph{circular word}, i.e., the set of all cyclic rotations of $w$.
We study the number of distinct square factors of the elements of $[w]$.
Amit and Gawrychowski first showed that this number is upper bounded by $3.14n$ in~\cite{AmitG17}. In a recent article~\cite{CMRRWZ26}, Charalampopoulos et al. improved this upper bound to $1.8n$ and conjectured that the sharp upper bound is $1.5n$. In this note, we improve this upper bound to $\frac{5}{3}n$.
\end{abstract}

\section{Introduction}

Repetitions in words constitute a central topic in combinatorics on words, with numerous connections to algorithms on strings and to the structural theory of periodicity. Among the various kinds of repetitions, the simplest and most studied are \emph{squares}, i.e., factors of the form $uu$ with $u$ a nonempty word. The earliest result concerning squares in words is obtained by Thue~\cite{Thue1906}, who proved the existence of square-free words of all lengths in three letters. The research concerning counting distinct squares in finite words was initiated by Fraenkel and Simpson. Let $w$ be a word of length $n$ and let $\sq(w)$ be the maximum number of distinct squares in $w$. In \cite{FraenkelS98},  Fraenkel and Simpson proved that $\sq(w) \leq 2n$ and conjectured that $\sq(w) \leq n$. After a series of improvements of the upper bound of $\sq(w)$, see, for example, $\sq(w) \leq 2n-O(\log n)$ due to Ilie~\cite{Ilie}; $\sq(w) \leq \frac{95}{48}n$ due to Lam~\cite{lam}; $\sq(w) \leq \frac{11}{6}n$ due to Deza, Franek and Thierry~\cite{DezaFT15} and $\sq(w) \leq 1.5n$ due to Thierry~\cite{thie},  the square conjecture was confirmed by Brlek the first author in \cite{sum}. 

The square-counting problem has also been studied in the context of \emph{circular} words. Currie~\cite{Currie02} first showed the existence of square-free circular words in three letters for all but a finite number of length exceptions using a computer-aided proof. After that, Shur~\cite{Shur10} gave a computer-free proof for the existence of exponentially many ternary square-free words, in terms of length, for given length. Let $w$ be a word of length $n$ and let $\sq([w])$ denote the maximum number of distinct squares in the circular word $[w]$. The first linear upper bound for $\sq([w])$ was obtained by Amit and Gawrychowski~\cite{AmitG17}, who proved that $\sq([w]) \leq 3.14n$ and for infinitely many $n$, there are words $w$ of length $n$ such that $\sq([w]) \geq 1.25n$. As a direct consequence of~\cite{sum} and~\cite{power}, one has the number of distinct square factors of a circular word $w$ of length $n$ is not larger than the number of squares in $w^2$, which induces $\sq([w]) \leq 2n$. In a recent article`\cite{CMRRWZ26}, Charalampopoulos et al. improved this upper bound to $\sq([w]) \leq 1.8n$ and showed that, for infinitely many $n$, there are words $w$ of length $n$ such that $\sq([w]) \geq 1.5n$. Charalampopoulos et al. also conjectured that $\sq([w]) \leq 1.5n$ is the sharp upper bound for all $w$ of length $n$~\cite{CMRRWZ26}. The goal of the present work is to improve the upper bound of $\sq([w])$. The main result of this note states the following.

\medskip
\noindent\textbf{Main Theorem}
\emph{Let $w$ be a finite word of length $n$. The number $\sq([w])$ of distinct square factors of the circular word $[w]$ is at most $\frac{5}{3}n$.}
\medskip

The paper is organized as follows. In Section~\ref{sec2} we recall some basic notions on words and graphs. In Section~\ref{sec3} we recall the definition of Rauzy graphs and small circuits and prove more properties of small circuits. In Section~\ref{sec4} we give a proof of our main theorem.

\section{Preliminaries}\label{sec2}

In this paper, we take basic terminology and notations about words from \cite{sum} and Lothaire \cite{lothaire3}. A {\em word} is a finite sequence   $w = w_1 w_2 \cdots w_{n}$ of \emph{letters} or symbols. The {\em length} $|w|$ of $w$ is $n$ and $w_i$ is the letter in  {\em position} $i$. The {\em concatenation of $w= w_1 w_2 \ldots w_{n}$ and $v= v_1 v_2 \ldots v_{m}$} is defined as 
$wv=w_1 w_2 \ldots w_{n}v_1 v_2 \ldots v_{m}$. The {\em alphabet} of the word $w$ is defined as $A= \Alphabet(w)=\left\{w_i \mid 1\leq i \leq n\right\}$. A word $u$ is called a {\em factor} of $w$ if $w = pus$ for some words $p,s$. The $i$-th prefix ending at position $i$ is denoted $w_p(i) =w_1w_2 \ldots w_{i}$ and the  $i$-th suffix starting at position $i$ is  
$w_{s}(i)=w_iw_{i+1} \ldots w_{n}$. Hence for word $w=w_1w_2 \ldots w_n$ and any integer $1 \leq i \leq n$, $w= w_p(i-1) w_{s}(i)$. The set of all length-$i$ factors of $w$ is denoted by $\fac_i(w)$ and the set of all factors of $w$ is denoted by $\fac(w)$. For any positive integers $j,k$ such that $j+k\le |w|$, define $w[j,j+k]=w_{j}\ldots w_{j+k}$ as a factor of $w$. 


Two finite words $u$ and $v$ are {\em conjugate} if there exist two words $x,y$ such that $u=xy$ and $v=yx$.
The circular word $[w]$ is the conjugacy class of $w$. Thus, $[w]=\left\{v|v=w_s(i)w_p(i-1), i=1,2,\ldots, n\right\}$.

For any positive integer $k$,  the {\em $k$-power} of a nonempty finite word $u$ is the concatenation of $k$ copies of $u$, and it is denoted by $u^k$. In particular, a \emph{square} is a word $w$ of the form $w=uu$. For any word $u$ and any rational number $\alpha=\frac{m}{|u|} \geq 1$, the \emph{$\alpha$-power} of $u$ is defined to be $u^au_0$ where $u_0$ is a prefix of $u$, $a=\lfloor \alpha \rfloor$ is the integer part of $\alpha$, and $|u^au_0|=m$. The $\alpha$-power of $u$ is denoted by $u^{\alpha}$. 

For all integers $m\in\NN$, define 
$$[w]_m=\fac_m(w^{\lfloor \frac{m}{|w|} \rfloor+1}).$$
In particular, if $m \geq |w|$, then
$$[w]_m=\left\{p^{\frac{m}{|w|}}| p \in [w]\right\};$$
if $m < |w|$, then $[w]_m$ is the set of all length-$m$ factors of $w^2$. 

A word $w$ is said to be {\em primitive} if it is not an integer power greater than or equal to $2$ of a word distinct from $w$. The set of primitive factors of $w$ is denoted by $\pri(w)$.

Let $\sq(w)$ and $\sq([w])$ be respectively the number of distinct nonempty squares in $w$ and $[w]$. Let 
$$\pow(w) =\{q\in\fac(w)| q=p^k,k\in\NN,k>1\}$$ 
be the set of power factors of $w$, and let $$\pow([w]) =\bigcup_{u \in [w]}\pow(u)$$ 
be the set of power factors of the circular word $[w]$. 

\begin{thm}[{\rm Fine and Wilf~\cite{lothaire1}}]
\label{three}
If a finite word $w$ has two different period $m,n$ and $|w| \geq m+n-\gcd(m,n)$, then $\gcd(m,n)$ is also a period of $w$.
\end{thm}


Next, we recall some basic definitions and properties concerning graphs mainly from Berge~\cite{berge}.

Let $G=(V,E)$ be a directed graph such that $V$ is the set of its vertices and $E$ is the set of its edges.
A {\em chain} of {\em length $k$} is a sequence of edges $e_1,e_2,\ldots,e_k$, such that for each $i$ satisfying $1 < i < k$, $e_i$ has one vertex in common with the preceding edge $e_{i-1}$ and another vertex in common with $e_{i+1}$. A {\em path} is a chain such that, for each $i$ satisfying $1 \leq i < k$, the terminal vertex of $e_i$ coincides with the initial vertex of $e_{i+1}$. A chain or a path is {\em closed} if it begins and ends at the same vertex. A {\em cycle} is a closed chain and a {\em circuit} is a closed path. Obviously, any circuit is a cycle, but a cycle may not be a circuit. 
A cycle  $C=(V,E)$ (or a circuit) is  {\em elementary} if each vertex on the chain occurs exactly twice: in this case, the  {\em length} of $C$ is $|V|=|E|$ .   

A graph $G$ is {\em weakly connected}, if for any pair of distinct vertices $(v_1,v_2)$, there exists a chain with edges $e_1,e_2,\ldots,e_k$ such that $v_1$ is an endpoint of $e_1$ and $v_2$ is an endpoint of $e_k$. For a weakly connected graph $G$ with $l$ edges and $s$ vertices,  the number $\chi(G)=l-s+1$ is  the {\em cyclomatic number} of $G$.

Now fix an order of the edges $E=\{e_1,\ldots,e_l\}$ and choose an orientation for each edge. Given a cycle $C$ in $G$, its \emph{vector-cycle} is the vector $\mu(C)=(c_1,\ldots,c_l)\in\mathbb{R}^l$ defined as follows:
for each $1\le i\le l$, let $r_i$ be the number of times $e_i$ occurs in $C$ with the chosen orientation, and let $s_i$ be the number of times it occurs against that orientation; then set $c_i=r_i-s_i$.
A family of cycles $C_1,\ldots,C_t$ is \emph{independent} if the vectors $\mu(C_1),\ldots,\mu(C_t)$ are linearly independent.
This notion does not depend on the initial choice of default orientations.

Let $\indep(G)$ be the maximum number of independent cycles in $G$. We call a set of $\indep(G)$ independent cycles in $G$ a \emph{basis} of $G$. 

A graph $M=(V', E')$ is called a \emph{subgraph} of $G$ if $V' \subset V$ and $E' \subset E$. In this case, let us denote $M \sqsubset G$.

\begin{thm}[{\rm Th. 2, Chap. 4 in~\cite{berge}}]
\label{book}
The cyclomatic number of a graph is the maximum number of independent cycles in this graph.
\end{thm}

\section{Rauzy Graphs and small circuits}\label{sec3}

Let us first recall the standard construction of Rauzy graphs introduced in~\cite{Rauzy}. Let $w$ be a word of length $n$. For any integer $i$ such that $1\leq i\leq n$, the Rauzy graph $\Gamma_i(w)$ is a directed graph whose set of vertices is $\fac_i(w)$ and the set of edges is $\fac_{i+1}(w)$.
An edge $e \in \fac_{i+1}(w)$ starts at the vertex $u$ and ends at the vertex $v$, if and only if $u$ is a prefix and $v$ is a suffix of $e$. It is "folklore" that the Rauzy graphs $\Gamma_i(w)$ are weakly connected for all $i$. A formal proof is given in \cite{aclss25}, and has already been mentioned in \cite{rote94}.

Let $\Gamma_i(w)$ be a Rauzy graph of $w$. An elementary circuit $C$ in $\Gamma_i(w)$ is called {\em small} if the length of $C$ is smaller than or equal to $i$.

For $1\le i\le n$, let $sc_i(w)$ be the number of small circuits in $\Gamma_i(w)$, and set
\[
sc(w)=\sum_{i=1}^n sc_i(w).
\]

We list and extend some useful propositions and notations from~\cite{sum}.

In~\cite{sum}, Brlek and the first author stated that, for any small circuit $C=(V,E)$ in $\Gamma_l(w)$, $1 \leq l \leq |w|$, there exists a primitive word $q$ such that $|q| \leq l$, $V=[q]_l$ and $E=[q]_{l+1}$. This fact can be extended to all circuits that are not necessarily small. Thus, for any circuit $C$ in $\Gamma_l(w)$, there exists a primitive word $q$ such that $C=([q]_l, [q]_{l+1})$, we denote it by $C(q,l)$.

\begin{lem}[{\rm Brlek and Li~\cite{sum}}]
\label{independent}
Let $w$ be a finite word. Then, 
\begin{itemize}
    \item[1] for all integers $1\leq i\leq |w|$,  the small circuits in $\Gamma_i(w)$ are independent. 
    \item[2] the total number of small circuits in $\Gamma(w)$ is upper bounded by $\indep(\Gamma(w))=|w|-|\Alphabet(w)|$.
\end{itemize}
\end{lem}
For any word $w$ with a primitive factor $p$, define {\em the class of $p$} to be $$\Class_p(w)=\left\{q^{k}| q \in [p], k \in \NN^+, k \geq 2, q^{k} \in \fac(w) \right\}.$$
Two classes $\Class_p(w)$ and $\Class_q(w)$ are equal if and only if $p$ and $q$ are conjugate. Moreover, $$\pow(w)=\bigcup_{p\in\prim(w)}\Class_p(w).$$ 
Let $\Class(w)=\{\Class_p(w)| p \in \prim(w)\}$.
For every class  $\Class_p(w)$ of $w$, define
\begin{eqnarray*}
E_p(w)&=&\Class_p(w) \cap \{u^{2i}| i \in \mathbb{N}^+, u \in \prim(w)\};\\
O_p(w)&=&\Class_p(w) \cap \{u^{2i+1}| i \in \mathbb{N^+},  u \in \prim(w)\}.
\end{eqnarray*}
Obviously, $E_p(w) \cup O_p(w) = \Class_p(w)$ and\\ $$\sq(w)=\sum_{\Class_p(w) \in \Class(w)} |E_p(w)|.$$ 

\begin{lem}\label{odd}

Let $p$ be a primitive factor of $w$ of length $l$. If $\Class_p(w) \neq \emptyset$, then, for all integers $1 \leq i \leq |\Class_p(w)| $, there is a small circuit $C(p,i+l-1)$ in the graph $\Gamma_{i+l-1}(w)$. Moreover, one has
\begin{equation*}
|O_p(w)| \leq |E_p(w)| \leq |O_p(w)|+l.
\end{equation*}
\end {lem}

\begin{proof}
The first part of the statement is from Lemma 10 of~\cite{power}, while the second part is from Lemma 15 of~\cite{word23}. To see the second part of the lemma, let $|\Class_p(w)|=rl+s$ with integers $r,s$ such that $r \geq 0$ and $0 \leq s \leq l-1$. The class $\Class_p(w)$ contains all $\{q^i| q \in [p], 2 \leq i \leq r+1\}$ and $s$ elements in $\{q^{r+2}| q \in [p],\}$.  If $r$ is even, then $|O_p(w)|=\frac{r}{2}l$ while $|E_p(w)|=\frac{r}{2}l+s$; and if $r$ is odd, then $|O_p(w)|=\frac{r-1}{2}l+s$ while $|E_p(w)|=\frac{r+1}{2}l$. \qed
\end{proof}

\begin{cor}\label{inj}
If $p$ is a primitive factor of $w$ of length $l$ and if $|\Class_p(w)|=t$, then, 
$$|O_p(w)| \geq \frac{t-l}{2}.$$
\end{cor}

Now, let us define the {\em split} of a circuit. Let $w$ be a word, $p$ be a primitive factor of $w$ and $m \in \N$, if $C(p,l)=\{[p]_{l},[p]_{l+1}\}$ is an elementary circuit when $C(p,l) \sqsubset \Gamma_l(w)$ for all $l >n$ and the path $C(p,m)=\{[p]_{m},[p]_{m+1}\}$ is not an elementary circuit in $\Gamma_m(w)$, then we say $C(p,.)$ {\em splits} at $m$. 

\begin{exa}
    
Let $p=abac,\ w=p^{\omega}$. Then for $i\ge 4$, $|[p]_i|=4$. Moreover, one has
\[
        [p]_3=\{aba,bac,aca,cab\},
\]
\[
        [p]_2=\{ab,ba,ac,ca\},
\]
and
\[
        [p]_1=\{a,b,c\}.
\]

For $i\ge 2$, $C(p,i)$ is an elementary circuit since for each of them, the edge set and vertex set both have 4 elements. While $C(p,1)$ is not elementary since the vertex set of it is $[p]_1$, and $|[p]_1|=3$. Instead, it decomposes into the two elementary circuits
\[
        a\to b\to a
\]
and
\[
        a\to c\to a.
\]
Their lengths are both equal to $2$, and hence
\[
        2+2=4=|p|.
\]
Therefore, in this example, $C(p,.)$ splits at
$1$ into two elementary circuits, see Figure~\ref{fig:split}.

\end{exa}

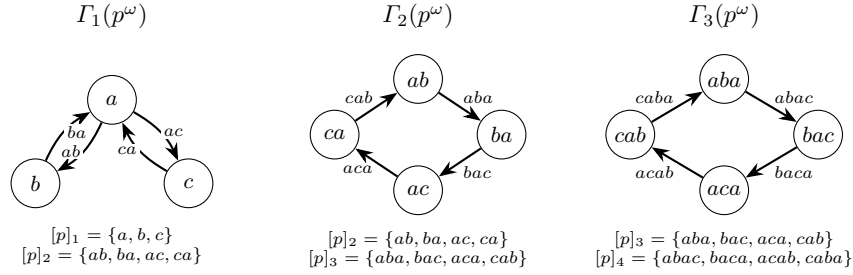
\begin{figure}[htbp]
\centering
\begin{tikzpicture}[
    >=Stealth,
    scale=0.92,
    every node/.style={transform shape},
    vertex/.style={
        draw,
        circle,
        minimum size=7mm,
        inner sep=1pt,
        font=\small
    },
    edge/.style={->,thick},
    edge label/.style={
        font=\scriptsize,
        midway,
        fill=white,
        inner sep=1pt
    }
]

\node[font=\normalsize] at (0,2.4) {$\Gamma_1(p^\omega)$};

\node[vertex] (a1) at (0,1.2) {$a$};
\node[vertex] (b1) at (-1.1,0) {$b$};
\node[vertex] (c1) at (1.1,0) {$c$};

\draw[edge,bend left=18] (a1) to node[edge label,left] {$ab$} (b1);
\draw[edge,bend left=18] (b1) to node[edge label,right] {$ba$} (a1);

\draw[edge,bend left=18] (a1) to node[edge label,right] {$ac$} (c1);
\draw[edge,bend left=18] (c1) to node[edge label,left] {$ca$} (a1);

\node[font=\scriptsize,align=center] at (0,-0.9) {
$[p]_1=\{a,b,c\}$\\
$[p]_2=\{ab,ba,ac,ca\}$
};

\begin{scope}[xshift=4.4cm]
\node[font=\normalsize] at (0,2.4) {$\Gamma_2(p^\omega)$};

\node[vertex] (ab2) at (0,1.5) {$ab$};
\node[vertex] (ba2) at (1.2,0.7) {$ba$};
\node[vertex] (ac2) at (0,-0.1) {$ac$};
\node[vertex] (ca2) at (-1.2,0.7) {$ca$};

\draw[edge] (ab2) to node[edge label,above right] {$aba$} (ba2);
\draw[edge] (ba2) to node[edge label,below right] {$bac$} (ac2);
\draw[edge] (ac2) to node[edge label,below left] {$aca$} (ca2);
\draw[edge] (ca2) to node[edge label,above left] {$cab$} (ab2);

\node[font=\scriptsize,align=center] at (0,-0.95) {
$[p]_2=\{ab,ba,ac,ca\}$\\
$[p]_3=\{aba,bac,aca,cab\}$
};
\end{scope}

\begin{scope}[xshift=8.8cm]
\node[font=\normalsize] at (0,2.4) {$\Gamma_3(p^\omega)$};

\node[vertex] (aba3) at (0,1.5) {$aba$};
\node[vertex] (bac3) at (1.35,0.7) {$bac$};
\node[vertex] (aca3) at (0,-0.1) {$aca$};
\node[vertex] (cab3) at (-1.35,0.7) {$cab$};

\draw[edge] (aba3) to node[edge label,above right] {$abac$} (bac3);
\draw[edge] (bac3) to node[edge label,below right] {$baca$} (aca3);
\draw[edge] (aca3) to node[edge label,below left] {$acab$} (cab3);
\draw[edge] (cab3) to node[edge label,above left] {$caba$} (aba3);

\node[font=\scriptsize,align=center] at (0,-0.95) {
$[p]_3=\{aba,bac,aca,cab\}$\\
$[p]_4=\{abac,baca,acab,caba\}$
};
\end{scope}

\end{tikzpicture}
\caption{$C(p,.)$ splits at 1 for $p=abac$.}
\label{fig:split}
\end{figure}

Here are two simple observations:

\begin{obs}\label{obs1}
If $C(p,.)$ {\em splits} at $m$, then $C(p,m)=\{[p]_{m},[p]_{m+1}\}$ is a joint graph of some elementary circuits $C_1,C_2, \ldots, C_k$, such that $C_i$ and $C_j$ cannot have any common edge if $i \neq j$. Otherwise $C(p,m+1)=\{[p]_{m+1},[p]_{m+2}\}$ will have two common vertices, so that $C(p,m+1)$ is not an elementary circuit. Consequently, $\sum_{i=1}^k|C_i|=|p|$.
\end{obs}

\begin{obs}\label{obs2}
If $C(p,m)$ is a small circuit, then it cannot {\em split} at $m-1$, since both of the sets $[p]_{m-1}$ and $[p]_{m}$ contain $|p|$ elements.
\end{obs}

\section{Proof of the main theorem}\label{sec4}

Let $w=w_1\ldots w_n$ be a word of length $n$ and let $W=w^2=W_1\ldots W_{2n}$. Let
\begin{eqnarray*}
&\sc'(W)=\{C(q,l)| \;|q|< \frac{n}{2}, q\in\prim(W), |q|\leq l\leq |\Class_q(W)|+|q|-1\};\\ 
&\sc(W)=\{C(q,l)| \;|q|< \frac{n}{2}, q\in\prim(W), |q|\leq l\}.    
\end{eqnarray*}
From the fact that $W$ is a square word, for any integer $1\leq m\le n+1$, $|\fac_m(W)|\leq n$. Moreover, since all factors of $W$ of length $m \leq n$ appears at least once at some position $i \leq n$ and the factor $W[1,m]=W[n+1,n+m]$ appear at least twice in $W$, the Rauzy graph $\Gamma_m(W)=([W]_m, [W]_{m+1})$ is a (not necessarily primitive) circuit. Moreover, one has the following lemma from~\cite{CMRRWZ26}:
\begin{lem}\label{multi}
If $w$ is non-primitive, then $\sq([w])\leq \frac{3n}{2}$.
\end{lem}

From now on, let us suppose that $w$ is primitive. In this case, $|\fac_n(W)|=|\fac_{n+1}(W)|=n$. Consequently, $\Gamma_n(W)$ is an elementary circuit. Moreover, there is no circuit in $\Gamma_l(W)$ for $l\geq n+1$.

Observing that if $u^2$ is a square of $[w]$, then there is a primitive word $v$ such that $u^2=v^{2i}$ for some positive integer $i$ and $|v|< \frac{n}{2}$, one has
$$\pow([w])\subseteq \pow^{'}(W):=\bigcup_{\substack{p\in\prim(w)\\|p|< \frac{n}{2}}}\Class_p(W).$$
From Lemma~\ref{odd}, there is a bijection from $\pow^{'}(W)$ to the set $\sc'(w)$. Moreover, one has
\begin{equation} \label{EQ1}
|\pow^{'}(W)| = |\sc'(W)| \leq |\sc(W)| \leq \indep(\Gamma(W)) \leq 2n.
\end{equation}
Consequently, one only needs to consider the small circuits of size smaller than $\frac{n}{2}$ in $\Gamma_1(W),\Gamma_2(W), \ldots, \Gamma_n(W)$.

\begin{lem}\label{Case1}
If for all $m\geq\frac{n}{2}$, all bases of $\Gamma_m(W)$ contain at least an elementary circuit $C$ of size $|C| \geq \frac{n}{2}$, then $\sq([w])\leq \frac{3n}{2}$.
\end{lem}

\begin{proof}
From the hypothesis, for all $m\geq\frac{n}{2}$, there is a circuit in $\Gamma_m(W)$ independent from the circuits in $\sc'(W)$. Thus,
$$|\sc'(W)|+\frac{n}{2} \leq |\indep(\Gamma(W))|\leq 2n.$$
Thus, from Equation~\ref{EQ1} 
$\sq([w]) \leq |\pow^{'}(W)| \leq |\sc'(W)| \leq \frac{3n}{2}$. \qed
\end{proof}

\begin{lem}\label{distinct}
If there exist a primitive word $p$ of length $l$ and a positive integer $k\geq 4$ such that $0<n-kl <l$ and $p^k \in \pow([w])$, 
then, for $n-l+1 \leq i \leq n$, there is a circuit of size larger than $\frac{n}{2}$ in all bases of $\Gamma_i(W)$.
\end{lem}

\begin{proof}
We first affirm that there exists $w'$ a conjugate of $w$, $p'$ a conjugate of $p$ and $p''$ a word of length $r=n-kl\neq 0$ such that $w'=p'^kp''$ and $p'[1] \neq p''[1]$. Otherwise, letting $v=p^kq$ be a conjugate of $w$ and $V=v^3$, for all $lk<i\leq 2n$, $V[i-lk,i+n-lk]$ must be a conjugate of $w$, and thus $V[i]=V[i-lk]$. Hence,  $v^2=p^{\alpha}$ for some $\alpha>2$, which leads to a contradict to primitivities of both $p$ and $v$ using Theorem~\ref{three}.

 Now, for $n-l+1 \leq i \leq n$, $\Gamma_i(W)=\Gamma_i(w'w'[1,i])$, and it is a closed path. Thus, there is an elementary circuit passing through the vertex $w'[1,i]$ and let us denote it by $C$. Define the function
$$f(i)=\begin{cases}
    i-l,& \; \text{if} \;i \leq kl;\\
    (k-1)l,&\; \text{if} \;i > kl.\\
\end{cases}$$
From the hypothesis that $k \geq 4$, one can easily check that $f(i)> \frac{n}{2}$ for all $n-l+1 \leq i\leq n$. 

We affirm $|C|\geq f(i)$. In fact, since $w'[1,i]$ is a vertex of $C$, there is a primitive word $q$ such that $|q|=|C|$ and that $w'[1,i]$ is a suffix of $w'[1,i]q$. Thus, there is an integer $j$ such that $w'[1,i]=w'[j,i]q$. If $|C|<f(i)$, we will prove that $w'[1,l]=w'[j,j+l]=p'$ and $j+l \leq kl$ in two cases. If $i \leq kl$,  one has $i-j+1=i-|q|> l$ and hence $w'[1,l]=w'[j,j+l]=p'$. Moreover, $j+l=|q|+l-1\leq i\leq kl$. Similarly, if $i > kl$, one has $i-j+1=i-|q|> i-(k-1)l=l+(i-kl)$ and one also has $w'[1,l]=w'[j,j+l]=p'$. Moreover, $j+l=|q|+l-1\leq (k-1)l+l\leq kl$. In all cases, $w'[1,l]=w'[j,j+l]=p'$ and $j+l \leq kl$. Consequently, $j-1=sl$ for some positive integer $s\leq k-1$ and $|q|=sl$. Since $w'[j,n]=w'[1,n-sl]$, $w'[kl+1,n]=w'[(k-s)l+1,n-sl]$ is a prefix of $p$, which contradicts the hypothesis that $w'[1] \neq w'[kl+1]$.

Now, let us prove that there exists only one circuit $C$ passing through the vertex $w'[1,i]$ in $\Gamma_i(W)$. Otherwise, the vertex $w'[1,i]$ must occur more than twice in $W'=w'w'[1,i]$. Obviously, the first and the last occurrences of $w'[1,i]$ appear respectively at the positions $W'[1]$ and $W'[n+1]$. If there exists a third occurrence of $w'[1,i]$ at some position $W'[j]$, then both $j-1$ and $n+1-j$ must be larger than or equal to the size of an elementary circuit passing through $w'[1,i]$. But all of these circuits (if there is more than one such circuit) are of size $|C|> \frac{n}{2}$, therefore, $$n=n+1-j+j-1> 2\frac{n}{2}=n,$$ which leads to a contradiction.

In conclusion, the circuit $C$ must appear in all bases of $\Gamma_i(W)$ and $|C| >\frac{n}{2}$. \qed
\end{proof}

\begin{lem}\label{Case2}
If there exists an integer $m\geq\frac{n}{2}$ such that 
\begin{itemize}
    \item[1] for all integers $m+1\leq i \leq n$, all bases of $\Gamma_i(W)$ contain at least an elementary circuit $C$ of size $|C| \geq \frac{n}{2}$,
    \item[2]$\Gamma_m(W)$ contains a circuit $C=C(p,m)$ with some primitive word $p$ such that $|p| \leq \frac{n}{4}$,
    \item[3]$\Gamma_i(W)$ does not contain the elementary circuit $C=C(p,i)$ for all integers $m+1 \leq i \leq n$,
\end{itemize}
 then $\sq([w])\leq \frac{13n}{8}$.
\end{lem}

\begin{proof}
 First, let $|\Class_p(W)|=t$. Since the first small circuit of the form $C(p,.)$ appears in the graph $\Gamma_{|p|}(W)$,one has $t \leq m-|p|+1$.
 From Corollary~\ref{inj}, 
 $$|O_p(W)|\geq \frac{t-|p|}{2}.$$ 
 Second, since $$\{C(p,l)| \;|p|+t+1 \leq l\leq m\} \subset \sc(W)\setminus \sc'(W),$$
one has, 
$$|\sc(W)\setminus \sc'(W)|\geq m-|p|-t.$$
Third, let us count $NC_{large}$ the number of Rauzy graphs containing circuits $C$ such that $|C| \geq \frac{n}{2}$ in all bases. Let $k,r$ be integers such that $0 \leq r=n-k|p|<|p|$, and from the hypotheses, $k \geq 4$. We first affirm that $t \leq (k-1)l$. Otherwise, from the definition of $\Class_p(W)$, there is $p'$ a conjugate $p$ such that $p'^{k+1} \in \fac([w])$, which contradicts the fact $n-k|p|<|p|$. Let us count $NC_{large}$ in two cases.\\

\textbf{Case I} If $n-2|p|-r < t \leq (k-1)|p|$, then $(k-2)|p|=n-2|p|< t \leq (k-1)|p|$, and thus, there exists $p'$ a conjugate of $p$ such that $p'^k \in \fac([w])$. From Lemma~\ref{distinct}, $NC_{large}\geq |p|$. In this case, $n-|p| \geq \frac{3n}{4}$ and thus,
\begin{align*}
  \sq([w]) &\leq 2n-\frac{t-|p|}{2}-|p|\\
  &=2n-\frac{t+|p|}{2}\\
  &\leq 2n-\frac{n-|p|}{2}\\
  &\leq \frac{13n}{8}.
\end{align*}

\textbf{Case II} If $ t \leq n-2|p|$, then from the definition of $m$, $NC_{large}\geq n-m$. In this case, 
\begin{align*}
  \sq([w]) &\leq 2n-\frac{t-|p|}{2}-(m-t-|p|)-(n-m)\\
  &=n+\frac{t+3|p|}{2}\\
  &\leq n+\frac{n+|p|}{2}\\
  &\leq \frac{13n}{8}.
\end{align*}

 Thus, in all cases, $\sq([w]) \leq \frac{13n}{8}.$ \qed
\end{proof}

\begin{lem}\label{Case3}
If there exists an integer $m\geq\frac{n}{2}$ such that 
\begin{itemize}
    \item[1] for all integers $m+1\leq i \leq n$, all bases of $\Gamma_i(W)$ contain at least an elementary circuit $C$ of size $|C| \geq \frac{n}{2}$,
    \item[2]$\Gamma_m(W)$ contains a circuit $C=C(p,m)$ with some primitive word $p$ such that $\frac{n}{4}\leq |p| \leq \frac{n}{3}$,
    \item[3]$\Gamma_i(W)$ does not contain the circuit $C=C(p,i)$ for all integers $m+1 \leq i \leq n$
\end{itemize}
 then $\sq([w])\leq \frac{5n}{3}$.
\end{lem}

\begin{proof}
 Let us do an analogous proof as in the previous case. Let $|\Class_p(W)|=t$, then $t \leq m-|p|$. From the same observation as above, $\{C(p,l)| \;|p|+t+1 \leq l\leq m\} \subset \sc(W)\setminus\sc'(W).$
Thus,
$$\sq([w]) \leq |\sc'(W)|-|O_p(W)| \leq |\sc(W)|-(m-t-|p|)- |O_p(W)|.$$

If $t>2|p|$, then there is a word $q$ conjugate to $p$ such that $q^4 \in \fac(w^{\prime})$ for some $w^{\prime}\in [w]$. However, since $\frac{n}{4}\leq |q| \leq \frac{n}{3}$ from the hypothesis, then $q^4$ is a conjugate of $w$, contradicting the primitivity of $w$. 
If $|p| \leq t \leq 2|p|$, then $|O_p(W)|=t-|p|$. Thus,
$$\sq([w]) \leq |\sc(W)|-|O_p(W)| \leq |\sc(W)|-(m-t-|p|)- (t-|p|)\leq \frac{5n}{3}.$$
If $ t < |p|$, then
$$\sq([w]) \leq |\sc(W)|-|O_p(W)| \leq |\sc(W)|-(m-t-|p|)\leq n+2|p|\leq \frac{5n}{3}.$$
In all cases,
$$\sq([w])\leq \frac{5n}{3}.$$ \qed
\end{proof}

\begin{proof}[of the Main Theorem]
First, if $w$ is non-primitive, the theorem has been proved in Lemma~\ref{multi}. 

Now, let us suppose that $w$ is primitive. In this case, $C(w,n) \sqsubset \Gamma_n(W)$ is an elementary circuit. If $C(w,.)$ never splits, or splits at some $m < \frac{n}{2}$, then one can conclude using Lemma~\ref{Case1}. 

If $C(w,.)$ splits at some $m \geq \frac{n}{2}$ into more than $2$ elementary circuits, then, from Observation~\ref{obs1}, there exists a circuit $C \sqsubset \Gamma_m(W)$ such that $|C|\leq \frac{n}{3}$. Then one can conclude using Lemma~\ref{Case2}, Lemma~\ref{Case3}. 

If $C(w,.)$ splits at some $m \geq \frac{n}{2}$ into $2$ elementary circuits, then, from Observation~\ref{obs1}, there exist $C_1=C(q_1,m),C_2=C(q_2,m)$ such that $|C_1|\geq \frac{n}{2}, \;|C_2|\leq \frac{n}{2}$. 

If $C(q_1,.)$ never splits, or splits at some $m' < \frac{n}{2}$, then, from Observation~\ref{obs2}, $C(q_2,.)$ cannot split before $\frac{n}{2}$, the two circuits $C(q_1,i),C(q_2,i)$ form a basis of $\Gamma_i(W)$ for $\frac{n}{2}\leq i \leq m$. Thus, one can conclude using Lemma~\ref{Case1}. 

If $C(q_1,.)$ splits at some $m' \geq \frac{n}{2}$, then, from the same argument as above, $C(q_1,i),C(q_2,i)$ form a basis of $\Gamma_i(W)$ for $m'\leq i \leq m$ with $|C(q_1,i)|\geq \frac{n}{2}$. Moreover, from Observation~\ref{obs1}, there exists a circuit $C \sqsubset \Gamma_{m'}(W)$ such that $|C|\leq \frac{n}{3}$. Thus, one can conclude using Lemma~\ref{Case2}, Lemma~\ref{Case3}. \qed

\end{proof}

\section{Conclusion}

In this note, we give a tighter upper bound of the number of distinct squares in a circular word of length $n$, which was previously $1.8n$. However, it is mentioned in~\cite{AmitG17} that, from computation, the sharp upper bound seems to be $1.5n$. Thus, the current upper bound may also be improved.

\bibliographystyle{splncs03}
\bibliography{biblio}

\end{document}